\documentclass[a4paper,10pt]{amsart}
\usepackage{amsfonts}
\usepackage{amsthm}
\usepackage{amssymb}
\usepackage{amsmath}
\usepackage{enumerate}
\usepackage{url}
\usepackage{color}

\newtheorem{theorem}{Theorem}

\newtheorem{lemma}[theorem]{Lemma}

\newcommand{\notdivides}[0]{\nmid}

\title{Normal numbers with digit dependencies}
\author{Christoph Aistleitner \and Ver\'onica Becher \and Olivier Carton}

\date{July 19, 2018}
\begin{document}

\maketitle

\begin{abstract}
  We give metric theorems for the property of Borel normality for real
  numbers under the assumption of digit dependencies in their
  expansion in a given integer base.  We quantify precisely how much digit
  dependence can be allowed such that, still, almost all real numbers are
  normal.  Our theorem states that almost all real numbers
  are normal when at least slightly more than $\log \log n$
  consecutive digits with indices starting at position $n$ are independent.
  As the main application, we consider the Toeplitz set $T_P$, which is the set of all sequences
  $a_1a_2 \ldots $ of symbols from $\{0, \ldots, b-1\}$ such that $a_n$ is
  equal to $a_{pn}$, for every $p$ in $P$ and $n=1,2,\ldots$. Here~$b$ is an integer base and~$P$ is a finite set of prime
  numbers. We show that almost every real number whose base $b$ expansion is in~$T_P$ is normal
  to base~$b$. In the case when $P$ is the singleton set $\{2\}$ we prove that more
  is true: almost every real number whose base $b$ expansion is in $T_P$ is
  normal to all integer bases. We also consider the Toeplitz transform
  which maps the set of all sequences to the set $T_P$ and we
  characterize the normal sequences whose Toeplitz transform is
  normal as well.
\end{abstract}
\bigskip

Mathematics Subject Classification  :  11K16,  68R15
\bigskip
\bigskip
\bigskip

\section{Introduction and statement of results} \label{sec:intro}

For a real number $x$ in the unit interval, its expansion in an integer
base $b \geq 2$ is a sequence of integers $a_1, a_2 \ldots$, where
$0\leq a_j<b$ for every $j$, such that
\begin{equation} \label{expansion}
x = \sum_{j=1}^{\infty} a_j b^{-j}. 
\end{equation}
We require that $a_j<b-1$ infinitely often to ensure that every number has
a unique representation.  When the base is understood we write
$x=0.a_1 a_2 \ldots$. 

The concept of normality of numbers was introduced by
Borel~\cite{Borel1909} in 1909 and there  are   several equivalent formulations
(see~\cite{BecherCarton18,Bugeaud12}).
The most convenient for our purposes was given by Pillai~\cite{Pillai1940}:
 A real number $x$ is simply normal to a given base
$b$ if every possible digit in $\{0, 1, \dots, b-1\}$ occurs in the $b$-ary
expansion of $x$ with the same asymptotic frequency (that is, with
frequency $1/b$).  A real number $x$ is normal to base $b$ if it is simply
normal to all the bases $b, b^2, b^3, \ldots$.  
Absolute normality is defined as normality to every integer base~$b \geq 2$.
 Borel proved that almost all numbers (with respect to Lebesgue measure) are absolutely normal. 


In this paper we consider normality under the assumption of additional
dependencies between the digits of a number. Let a base $b$ be fixed and
consider the set of real  numbers $x=0.a_1 a_2 a_3 a_4 \dots$ in the unit interval 
where the  digits $a_1, a_2, a_3,  a_4\dots$ can be divided into \emph{free} or
\emph{independent} digits on the one hand, and \emph{dependent} digits on
the other hand.  The free digits can be chosen at will, while for the
dependent digits there is a restriction  which prescribes their values deterministically 
from the values of a certain set of digits with smaller indices. 
For example, consider the restriction that the equality
\[
a_{2n} = a_n
\]
must hold for all $n \geq 1$; then $a_1, a_3, a_5, \dots$ are independent
and can be freely chosen, while $a_2,a_4,a_6,\dots$ are dependent since
they are completely determined by earlier digits.

A special form of such digit dependencies has been formalized by Jacobs and
Keane in~\cite{JacobsKeane69} by considering \emph{Toeplitz sequences} and
the \emph{Toeplitz transform},
which we present now.  Fix an integer $b \geq 2$. Let $A$ denote the
alphabet $A = \{0, \dots, b-1\}$, and write $A^\omega$ for the set of all
infinite sequences of symbols from $A$.
For a  positive integer $r$  and a   set $P=\{p_1, \ldots, p_r\}$ of $r$ prime numbers,
 let $T_P$ be the set of all Toeplitz sequences, that is, the set of all sequences 
$t_1t_2t_3\cdots$ in~$A^\omega$ such that 
for every   $n \geq 1$  and for every $i=1, \ldots , r$, 
\begin{equation} \label{digit_res}
t_n = t_{np_i}.
\end{equation}

The Toeplitz transform maps sequences in $A^\omega$ to the Toeplitz
set $T_P$.  Let $j_1, j_2 ,j_3, \ldots$ be the enumeration in increasing
order of all positive integers which are not divisible by any of the primes
$p_1, \ldots, p_r$.  Then, every positive integer~$n$ has a unique
decomposition $n = j_k p_1^{e_1} \cdots p_r^{e_r}$, where each
integer~$e_i$ is the $p_i$-adic valuation of~$n$, for $i=1,\ldots,r$.  The
Toeplitz transform $\tau_P:A^\omega\to A^\omega$ is defined as
\[
\tau_P( a_1 a_2 a_3\ldots) = t_1 t_2 t_3 \ldots\
\]
where 
\begin{equation} \label{t_n}
t_n = a_k \quad \text{when $n$ has the decomposition } \quad n = j_k p_1^{e_1} \cdots p_r^{e_r}.
\end{equation}
Thus, the image of $A^\omega$ under the transform $\tau_P$ is the set
$T_P$.  Since elements of $A^\omega$ can be identified with real numbers in $[0,1]$ in a natural way via \eqref{expansion}, the transform $\tau_P$ induces a transform $[0,1] \mapsto T_P$, which we denote by $\tau_P$ as well. We endow $T_P$ with a probability measure $\mu$, which is the
forward-push by $\tau_P$ of the uniform probability measure~$\lambda$ on
$A^\omega$ (which, in turn, is the infinite product measure generated by
the uniform measure on $\{0,\ldots, b-1\}$).  Note that $\mu$ can also be
seen as a measure on the set of all sequences, that is on $A^\omega$, since
$T_P$ is embedded in $A^\omega$.  Again, as already noted above for $\tau_P$, 
by identifying infinite sequences with real numbers,
the measure $\mu$ on $A^\omega$ also induces a measure on $[0,1]$, which we denote by
$\mu$ as well. 
For any measurable set~$X \subseteq T_P$, $\mu(X) = \lambda(\tau_P^{-1}(X))$. Informally speaking, $\mu$ is the natural uniform measure on
the set of all sequences (resp.,\ real numbers) which respect the
digit dependencies imposed by \eqref{digit_res}.

The Toeplitz transform $\tau_P$ also induces a function
$\delta:~\mathbb{N} \mapsto \mathbb{N}$ on the index set, by defining
$\delta(n) = k$ for $k$ and $n$ as in \eqref{t_n}. 
Hence, $\delta(n) = k$ means that the $n$-th symbol $t_n$ of the image of
$a_1 a_2 a_3 \dots$ under the Toeplitz transform is $a_k$. This can also be
written as
$$
t_1 t_2 t_3 \cdots = \tau_P (a_1 a_2 a_3 \cdots) = a_{\delta(1)} a_{\delta(2)} a_{\delta(3)} \cdots.
$$
The $n$-th symbol $t_n(x)$ of $\tau_P(x)$ is a measurable function
$[0,1] \mapsto \{0,\dots,b-1\}$. 
Thus, $t_n(x)$ is random
variable on the space $([0,1],\mathcal{B}(0,1),\lambda)$. Since
$t_n(x) = a_{\delta(n)}$ for all $n$, it is easy to see that two random
variables $t_m$ and $t_n$ are independent (with respect to both measures
$\lambda$ and $\mu$) if and only if $\delta(m) \neq \delta(n)$, that is, if
they do not originate in the same digit of $x$ by means of the Toeplitz
transform.

We say that an infinite sequence of symbols from $\{0, \ldots , b-1\}$ is normal if it is the expansion of 
a real number which is normal to base~$b$.
Our first result is the following theorem. It shows that ``typical''
elements of $T_P$ are normal, just as by Borel's theorem ``typical'' real
numbers are normal.  Thus, imposing additional digit dependencies does not
destroy the fact that almost all numbers are normal.

\begin{theorem} \label{thm:normal}
  Let $b \geq 2$ be an integer, and let~$P$ be a finite set of primes.  Let
  $\mu$ be the ``uniform'' probability measure on the set $T_P$, defined
  above.  Then, $\mu$-almost all elements of~$T_P$ are the expansion in base~$b$ of a normal number.
\end{theorem}

The proof of Theorem~\ref{thm:normal} generalizes the one given by
Alexander Shen (personal communication, June 2016) for the special case
$P = \{ 2\}$. It relies on the fact that the sequence of all integers which
are generated by a finite set $P$ of primes, sorted in increasing order,
grows very quickly. More precisely, there is a very strong gap condition
which gives a lower bound for the minimal size of the gap $k - k'$ whenever
$k>k'$ are two numbers generated by $P$, as a function of $k'$. In the case
when $P=\{p\}$ is a singleton this is a trivial observation, since then the
integers generated by $p$ form a geometric progression, but when $P$ has
cardinality at least~$2$ this is a subtle property, which can be
established using Alan Baker's celebrated theory of linear forms of
logarithms (see for example~\cite{baker}). Note that the set of numbers
generated by finitely many primes forms a semi-group; it can be extended to
a group of what is called S-units, which are well-studied objects in
algebraic number theory due to their connection with the theory of
Diophantine equations.
%
It is a well-known fact in probabilistic number theory that 
 parametric lacunary sequences,
which are sequences  of the form $(m_n x)_{n\geq 1}$ where there exists $c>1$ such that 
$m_{n+1}/m_n >1$ for all $n$, often lead to weakly dependent random
systems that behave asymptotically like truly independent random systems;
see for example~\cite{ai} for a survey.

The classical lacunary systems originate from geometric progressions, but it turned out that often it is
possible to adapt the machinery to sequences of integers generated by
finitely many primes. Classical results in that direction are denseness
properties which Furstenberg~\cite{furst} deduced from the disjointness of
corresponding measure-preserving transforms, and Philipp's~\cite{phil}
law of the iterated logarithm. However, the setting in the present paper is
very different from these earlier results.
By Weyl's criterion, 
normality of a number $x$ to base $b$ is equivalent to the fact that the sequence of
fractional parts of $x,bx,b^2x,\dots$ is uniformly distributed modulo~$1$. 
Now our purpose is not to replace $(b^n)_{n \geq 1}$ by a sequence of
integers generated by finitely many primes (as it is in~\cite{furst,phil}),
but instead we impose a restriction on the digits of $x$ for special sets
of indices, and try to prove that for such $x$ we have uniform distribution
mod~$1$ of the fractional parts of $x,bx,b^2x,\dots$. These are very
different problems and completely different methods are required.

Theorem~\ref{thm:normal} shows that $\mu$-almost all numbers are normal in
a given base $b$. However, it turns out that much more is true, at least in
the case when $P = \{2\}$. For real numbers satisfying $a_{2n}=a_n$ for
every $n\geq 1$ in their expansion in base~$b$, we prove that almost surely
they are actually absolutely normal.  This is the same as saying that
$\mu$-almost all real numbers in $[0,1]$ are absolutely normal. This result
is in the following theorem.

\begin{theorem}\label{thm:schmidt}
  Let $b \geq 2$ be an integer, let $P=\{2\}$ and let $\mu$ be the
  ``uniform'' probability measure on $T_P$.  Then, $\mu$-almost all
  elements of $T_P$ are the expansion in base $b$ of an absolutely normal
  number.
\end{theorem}

To prove Theorem~\ref{thm:schmidt} we adapt the work of
Cassels~\cite{Cassels1959} and Schmidt~\cite{Schmidt1961}. 
Our argument is also based on the idea of giving upper bounds for certain Riesz products,
although the setting is quite different.
 For example, Cassels worked on a
Cantor-type set of real numbers whose ternary expansion avoids the digit
$2$ (and which therefore cannot be normal to base $3$), and he had to
establish certain regularity properties of the uniform measure supported on
this fractal set.  In contrast, we clearly have to deal with the measure
$\mu$ which is the uniform measure on the set of real numbers which respect
the digit restriction~\eqref{digit_res}.  This property on digits is more
delicate than that of avoiding certain digits altogether, but still it
turns out that it is possible to use similar techniques.  Our proof of
Theorem~\ref{thm:schmidt} can in principle be generalized to Toeplitz
sequences for arbitrary finite sets of primes~$P$ instead of $P = \{2\}$.
However, to keep the proof reasonably simple we do not deal with this
general setting in the present paper.
The proof  works by partitioning the set of all possible positions of symbols
into equivalence classes such that at all these positions of the Toeplitz sequence the same symbol occurs; using notation introduced above, each equivalence class collects all those indices for which the function $\sigma$ gives the same value. 
 In case  $P=\{2\}$, all Toeplitz sequences satisfy that 
all the positions of the form $2^n$, for $n=0,1,\ldots$ have the same symbol,
all the positions of the form $3\ 2^n$, for $n=0,1,\ldots$ have the same symbol,
all the positions of the form $5\ 2^n$, for $n=0,1,\ldots$ have the same symbol, and so on.
This determines, for each odd number, one equivalence class  of positions.
In case $P$ is a finite set of prime numbers,  $P=\{p_1, p_2, \ldots , p_r\}$,
 the  definition of the equivalence classes is subtler. 
For each positive integer $s$ that is not a multiple of any $p$ in $P$, 
all the positions of the form $s p_1^{n_1} p_2^{n_2} \ldots p_r^{n_r}$
where  each $n_1, n_2 \ldots$ take the values $0,1,2,\ldots$  form an equivalence class. 
The proof of Theorem~\ref{thm:schmidt} generalized to a set~$P$ with 
finitely many primes requires a bound exponential sums over all the 
equivalence classes. 

We are also interested in a general framework for results of the form of
Theorem~\ref{thm:normal}, where, however, the digit dependencies can be
much more general than those imposed for Toeplitz sequences.  How much
digit dependence can be allowed in some given base such that, still, almost
all real numbers are normal to that base? Our Theorem~\ref{thm:theorem_D}
below quantifies how many consecutive digits have to be independent, in
order to keep the usual property that almost all numbers are normal. Quite
surprisingly, it turns out that only a very low degree of independence is
necessary. The theorem says, roughly speaking, that as long as we can
assure that slightly more than $\log \log n$ consecutive digits with
indices starting at $n$ are independent for all sufficiently large~$n$,
then almost all real numbers are normal. On the other hand, assuming
independence of blocks of $\log \log n$ consecutive digits is not
sufficient.

For the statement of the following theorem, let
$(\Omega,\mathcal{A},\mathbb{P})$ be a probability space.  Let
$X_1, X_2, \dots$ be a sequence of random variables (that is, measurable
functions) from $(\Omega,\mathcal{A},\mathbb{P})$ into $\{0, \dots, b-1\}$.

 \begin{theorem} \label{thm:theorem_D}
   Assume that for every $n \geq 1$ the random variable $X_n$ is uniformly
   distributed on $\{0, \dots, b-1\}$. Assume furthermore that there exists
   a function $g:~\mathbb{N} \mapsto \mathbb{R}$ which is monotonically
   increasing to $\infty$ such that for all sufficiently large $n$ the
   random variables
\begin{equation} \label{indep}
X_n, ~X_{n+1},~\dots,~X_{n+ \lceil g(n) \log \log n \rceil}
\end{equation}
are mutually independent. Let $x$ be the real number whose expansion in
base $b$ is given by $x=0.X_1 X_2 X_3 \dots$. Then $\mathbb{P}$-almost
surely the number $x$ is normal to base~$b$.

On the other hand, for every base $b$ and every constant $K>0$ there is an
example where for every $n \geq 1$ the random variable $X_n$ is uniformly
distributed on $\{0, \dots, b-1\}$ and where for all sufficiently large $n$
the random variables
\begin{equation} \label{disjoint}
X_n,~ X_{n+1}, ~\dots, ~X_{n+\lceil K \log \log n \rceil}
\end{equation}
are mutually independent, but $\mathbb{P}$-almost surely the number $x = 0.X_1 X_2 X_3 \dots$ even fails to be simply normal.
\end{theorem}

Note that the theorem gives an optimal condition for the degree of
independence which is necessary to have normality for ``typical''
numbers. From the proof of Theorem~\ref{thm:theorem_D} it is visible that
for the correctness of the conclusion of the first part of the theorem it
is not necessary that~\eqref{indep} holds for \emph{all} (sufficiently
large) $n$, but that it is possible to allow a set of exceptional indices
which has to be assumed to have small density in some appropriate
quantitative sense. However, to keep the presentation short we do not state
the theorem in such generality.

Theorem~\ref{thm:normal} says that $\mu$-almost all sequences are 
mapped by  the  Toeplitz transform  $\tau_P$ to normal ones.
A natural question is  whether  all the  normal sequences are mapped by   $\tau_P$ to normal ones.
The following example shows that this is not the case. 
Let  $P = \{2\}$  and let $x = a_1a_2a_3 \cdots$ be a normal sequence such that $a_{2n} = a_n$ for
each $n \ge 1$.
Then, $\tau_P(x) = d_1d_2d_3\cdots$ satisfies
for each $n \ge 1$, $d_{2n} = d_n$ and $d_{2n-1} = a_n$.  Combining these
relations we obtain $d_{4n-2} = d_{2n-1} = a_n = a_{2n} = d_{4n-1}$, which
proves that $\tau_P(x) $ is not normal.  This example actually shows that
applying twice the transform~$\tau_P$ to a normal sequence never yields a
normal sequence.  The last theorem of this paper,
Theorem~\ref{thm:toeplitz}, gives a characterization of those normal
sequences whose Toeplitz transform is also normal.

We refer to finite sequences  of symbols in $\{0,\ldots , b-1\}$ as words.
If $u$ is a word  we write $|u|$  to denote its length.

 

For simplicity, the next theorem is only stated and proved when the
cardinality of the set~$P$ of prime numbers is~$2$, but it can be
generalized to any finite set $P = \{ p_1, \ldots, p_r\}$ of prime numbers.
\begin{theorem} \label{thm:toeplitz}
  Let $P = \{p_1, p_2 \}$ where $p_1$ and $p_2$ are two prime numbers.  For
  a sequence~$x$ in $\{0,\ldots , b-1\}^\omega$, the following conditions
  are equivalent:
  \begin{enumerate}[(I)] \itemsep=0cm
  \item \label{toeplitz-i}
    The Toeplitz transform $\tau_P(x)$ of~$x$ is normal to base $b$.
 
  \item \label{toeplitz-ii}
    For every integer~$k \ge 0$ and every family
    $(u_{i_1,i_2})_{0 \le i_1,i_2 \le k}$ of words of arbitrary length, the
    limit
    \begin{displaymath}
      \lim_{N\to\infty} \frac{1}{N}\#\left\{ n :
      \begin{array}{l}
        1\leq n \le N, \text{ and for all } 0 \le i_1,i_2 \le k,\\
        u_{i_1,i_2} \text{ occurs in $x$ at position $(p_1-1)(p_2-1)p_1^{i_1}p_2^{i_2}n$}
      \end{array} 
      \right\}
    \end{displaymath}
    is equal to $b^{ -\sum_{0 \le i_1,i_2 \le k} |u_{i_1, i_2}|}$.
\end{enumerate}
\end{theorem}
Note that Condition~(\ref{toeplitz-ii}) with $k=0$ just states that the
sequence~$x$ itself must be normal.  This is indeed required because the
symbols in~$\tau_P(x)$ at positions not divisible by~$p_1$ and~$p_2$ are
exactly those in~$x$, in the same order.  Condition~(\ref{toeplitz-ii}) can
be viewed as a sort of asymptotic probabilistic independence between words
occurring at positions of the form~$(p_1-1)(p_2-1)p_1^{i_1}p_2^{i_2}n$.  It
states that the asymptotic frequency of a family of words is exactly
the product of their frequencies.

We hope that Theorem \ref{thm:toeplitz} will help to find a construction of
a normal sequence in $T_P$ for some general finite set $P$ of primes.  A
construction of one explicit normal sequence in $T_P$ for $b=2$ and the
special case $P=\{2\}$ appears in~\cite{BecherCarton18}
and~\cite{BecherCartonHeiber18}.  This construction can be generalized to
any integer base $b$ and any singleton~$P$.  
\bigskip

The remainder of the paper is devoted to the proofs of
Theorems~\ref{thm:normal} to~\ref{thm:toeplitz}.

\section{Proof of Theorem~\ref{thm:normal} } \label{sec:normal}

We prove Theorem~\ref{thm:normal} by showing that it is a consequence of
Theorem~\ref{thm:theorem_D}, together with number-theoretic results of
Tijdeman~\cite{Tijdeman73}. 
Let $r$ be a positive integer and
let~$P= \{p_1, \ldots , p_r\}$ be a set of $r$ primes. We define the sets
$K$ and $L$ as
\begin{displaymath}
  K = \{ p_1^{e_1} \cdots p_r^{e_r} : e_i \ge 0 \} 
  \quad\text{and}\quad
  L = \{ \ell : p_i \notdivides\ell,  i=1,\ldots, r\}.
\end{displaymath}
Thus, every positive integer~$n$ can be written in a unique way as
$n = k\ell$ for some $k \in K$ and $\ell \in L$.  We define a equivalence
relation~$\sim$ on the set of positive integers by writing $n \sim n'$
whenever there are $k,k' \in K$ and $\ell \in L$ such that $n = k\ell$ and
$n' = k'\ell$.

\begin{lemma} \label{lem:Tijdeman}
  There exists an integer~$n_0$ such that if $n' \sim n$ and $n'>n>n_0$, then
  $n'-n > 2\sqrt{n}$.
\end{lemma}
\begin{proof}[Proof of Lemma~\ref{lem:Tijdeman}]
In~\cite{Tijdeman73} Tijdeman proved that there exists a positive constant~$C$ 
such that for all  $k,k' \in K$ satisfying $k < k'$ we have
\[
k'-k > k/(\log k)^C.
\]
There also exists an integer $k_0$ such that
\[
k/(\log k)^C > 2\sqrt{k}
\]
for all $k \geq k_0$.   
 Let $n_0 = 4k_0^2$.  Suppose that $n' = k'\ell$ and $n = k\ell$ for
  $k,k'\in K$ and $\ell \in L$.
  If $k \ge k_0$, then \[
n'-n = (k'-k)\ell > 2\sqrt{k}\ell > 2\sqrt{n}.
\]
  If $k < k_0$ and $n = k\ell > 4k_0^2$, then $\ell > 4k_0$.  It follows
  that 
\[
n'-n = (k'-k)\ell > \ell > 2\sqrt{\ell}\sqrt{k_0} > 2\sqrt{n}.
\]
\end{proof}

\begin{proof}[Proof of Theorem~\ref{thm:normal}]
  Let $x = 0.a_1a_2a_3 \dots$ be a real number and let
  $\tau_P(x) = 0.t_1 t_2 t_3 \dots$ be its Toeplitz transform. As noted in
  the introduction, $t_n(x)$ is a measurable function from
  $([0,1],\mathcal{B}(0,1),\lambda)$ to $\{0, \dots, b-1\}$ for all
  $n$. Clearly $t_n$ has uniform distribution on $\{0, \dots, b-1\}$, since
  $t_n(x) = a_{\delta(n)}$ and the digit $a_{\delta(n)}$ takes all possible
  values with equal probability with respect to Lebesgue
  measure. Lemma~\ref{lem:Tijdeman} can be rephrased as saying that for all
  sufficiently large $n$ all the numbers
$$
\delta(n) , \delta(n+1), \dots, \delta(n + \lfloor 2 \sqrt{n} \rfloor)
$$
are different, since $n \sim n'$ holds if and only if
$\delta(n) = \delta(n')$. Thus, for all sufficiently large $n$, the random
variables
$$
a_{\delta(n)}, a_{\delta(n+1)}, \dots, a_{\delta(n + \lfloor 2 \sqrt{n} \rfloor)}
$$
are mutually independent with respect to $\lambda$, since different digits
of a real number are mutually independent with respect to Lebesgue measure
(the digits are Rademacher random variables; their independence with
respect to $\lambda$ was first observed by Steinhaus,
see~\cite{ai}). However, this is the same as saying that
$$
t_n, t_{n+1}, \dots, t_{n + \lfloor 2 \sqrt{n} \rfloor}
$$
are mutually independent for sufficiently large $n$, with respect to
$\lambda$.  Thus, we see that the assumptions of
Theorem~\ref{thm:theorem_D} are satisfied, and thus for $\lambda$-almost
all input values $x$, the number $0.t_1 t_2 t_3 \dots = \tau_P(x)$ is normal
to base $b$. 
From the way $\mu$ is obtained from $\lambda$, this is
equivalent to saying that $\mu$-almost all sequence in $T_P$ are normal to
base $b$, which proves Theorem~\ref{thm:normal}.
\end{proof}

\section{Proof of Theorem~\ref{thm:schmidt}}

Fix the integer  $b \geq 2$  and $P=\{2\}$. 
We need to show that for all integers $r \geq 2$,
 $\mu$-almost all elements of $T_P$ are the expansion of a number that is normal to base~$r$.
As usual, we say that two positive integers are {\em multiplicatively dependent} if 
one is  a rational power of the other.
In case $b$ and $r$ are multiplicatively dependent,  Theorem \ref{thm:schmidt} 
follows immediately from  Theorem \ref{thm:normal} because normality to base $b$ is equivalent to normality 
to any multiplicatively dependent base~$r$.

In case $r$ is multiplicatively independent to $b$ the main structure of our proof follows the work of Cassels
in~\cite{Cassels1959}, but adapted to the uniform measure on the real
numbers whose expansion is in $T_P$. We need two lemmas. The first one,
Lemma~\ref{lemma:cosine}, is similar to
Schmidt's~\cite[Hilfssatz~5]{Schmidt1961}, except that in our case the
product is taken only over the odd integers.  The second one,
Lemma~\ref{lemma:schmidt}, bounds the $L^2(\mu)$ norm of the appropriate
exponential sums.

We start by introducing some notation.
For $v=v_1v_2 \ldots$ in $T_P$ let $x_v$ be the real number  in the unit interval defined by
\begin{equation} \label{def:xv}
x_v=\sum_{j\geq 1} b^{-j} v_j.
\end{equation}
We write   $T_P(\ell)$ for the set of  sequences  of length $\ell$ that are 
initial segments  of elements in the  Toeplitz set 
 $T_P$ for $P=\{2\}$, that is,
\[
T_P(\ell)= \Big\{  a_1 a_2\ldots a_\ell \in \{0, \ldots, b-1\}^\ell :  a_n=a_{2n}  \text{ for each } 1\leq n\leq \ell/2 \Big\}.
\]
Similar to \eqref{def:xv}, for  $v=v_1v_2 \ldots v_\ell$ in $T_P(\ell)$ we  let $x_v$ be 
\[
x_v=\sum_{j= 1}^\ell b^{-j} v_j.
\]

\begin{lemma}\label{lemma:cosine}
 Let $r$ and $b$ be multiplicatively independent positive integers.  There is a
  constant $c>0$, depending only on $r$ and $b$, such that for all positive
integers $J$ and $L$ with~$L\geq b^{J}$, and for every positive integer $N$,
\begin{equation}
\nonumber
  \sum_{n=0}^{N-1}
\prod_{\substack{q=J+1\\q\text{ odd}}}^\infty 
\left( \frac{1}{b} + \frac{b-1}{b} \left|\cos\left(\pi r^n  L  b^{-q }\right)\right| \right)
\leq 2 N^{1-c}.    
  \end{equation}
\end{lemma}
\begin{proof}
  Schmidt's~\cite[Hilfssatz~5]{Schmidt1961} states that for all
  multiplicatively independent integers $r \geq 2$ and $s \geq 2$ there is
  a constant $c_1>0$, depending only on $r$ and $s$, such that for all
  positive integers $K$ and $L$ with
  $L\geq s^{K}$, and for every positive integer $N$,\setcounter{footnote}{1}\footnote{In Schmidt's paper the
    sum is written as $\sum_{r=0}^{N-1}$, but actually the sum is
    $\sum_{n=0}^{N-1}$.} 
\begin{equation*}
\sum_{n=0}^{N-1}  \prod_{k=K+1}^\infty |\cos ( \pi r^n L  s^{-k}) | \leq 2 N^{1-c_1}.
\end{equation*}
When examining the proof of this \emph{Hilfssatz}, one sees that the only
properties of the function $|\cos (\pi x)|$ that are used in the proof are
the periodicity, the fact that $|\cos(\pi x)| \leq 1$, and finally the fact
that $|\cos (\pi/s^2)| < 1$. However, all these properties also hold for
the function
$\frac{1}{p} + \frac{p-1}{p} \left|\cos\left(\pi x\right)\right|$ for any integer $p>1$, so
Schmidt's proof can also be used without any further changes to show that
\begin{equation}\label{schmidt}
\sum_{n=0}^{N-1}  \prod_{k=K+1}^\infty \left( \frac{1}{p} + \frac{p-1}{p} \left|\cos\left(\pi r^n  L  s^{-k }\right)\right| \right) \leq 2 N^{1-c_2},
\end{equation}
for a constant $c_2>0$ depending only on $p$, $r$ and $s$.

Our lemma assumes that $r$ and $b$ are multiplicatively independent, so $r$
and~$b^2$ are multiplicatively independent as well.  Replacing $p$ by $b$ and $s$ by $b^2$
in (\ref{schmidt}), we obtain
\[
\sum_{n=0}^{N-1}  \prod_{k=K+1}^\infty \left( \frac{1}{b} + \frac{b-1}{b} \left|\cos\left(\pi r^n  L  b^{-2k }\right)\right| \right)  \leq 2N^{1-c_3},
\]
for all $K,L$ satisfying $L \geq (b^2)^K$, where the constant $c_3>0$
depends on $r$ and~$b^2$ (which is equivalent to saying that $c_3$ depends
on $r$ and $b$).  In particular this holds for all $L$ which are multiples
of $b$, satisfying $L\geq b^{2K}$. So let us assume that $L$ is a multiple
of $b$, and that accordingly $L = bm$. Then we have
\[
\sum_{n=0}^{N-1}  \prod_{k=K+1}^\infty \left( \frac{1}{b} + \frac{b-1}{b} \left|\cos\left(\pi r^n  m  b^{-(2k-1) }\right)\right| \right) \leq 2N^{1-c_3}.
\]
provided that $bm \geq b^{2K}$. 
Now writing $2k-1=q$  this is
\[
\sum_{n=0}^{N-1}  \prod_{\substack{q=2K+1,\\q \text{ odd}}}^\infty \left( \frac{1}{b} + \frac{b-1}{b} \left|\cos\left(\pi r^n  m  b^{-q}\right)\right| \right) \leq 2N^{1-c_3}.
\]
Finally writing $J+1 = 2K+1$ this is
\[
\sum_{n=0}^{N-1}  \prod_{\substack{q=J+1,\\q \text{ odd}}}^\infty \left( \frac{1}{b} + \frac{b-1}{b} \left|\cos\left(\pi r^n  m  b^{-q}\right) \right| \right) \leq 2N^{1-c_3},
\]
which holds for $b m \geq b^{2K}$, that is (since $K = J/2$) for $b m \geq (b^2)^{J/2}$, 
or equivalently for $m \geq b^{J-1}$. We can relax the final restriction to $m \geq b^J$.
 This proves the lemma.
\end{proof}

As usual, we write $e(x)$ to denote ${\rm e}^{2\pi i x}$.  

\begin{lemma}\label{lemma:schmidt}
  Let $b \geq 2$ be a integer, and assume that $P=\{2\}$. Let $T_P$ be the
  corresponding Toeplitz transform in base $b$, and let $\mu$ be the
  associated measure, as introduced in Section~\ref{sec:intro}. Let
  $r \geq 2$ be an integer 
multiplicatively independent to~$b$.
  Then for all integers $h \geq 1$ there exist constants $c>0$ and
  $k_0 >0$, depending only on $b,r$ and $h$,  such
  that  for all positive integers $k,m$ satisfying
$m \geq k + 1 + 2 \log_r b \geq k_0$,
$$
\int_0^1 \left| \sum_{j=m+1}^{m+k} e(r^j h x) \right|^2 d\mu(x) \leq k^{2-c}.
$$
\end{lemma}


\begin{proof}
We write $\ell$ for the smallest even integer which is larger than 
$$
((m+k+1) \log_b r) + \log_b h.
$$
Let $x \in [0,1]$ be given, and let $\bar{x}$ be the same number as $x$,
but with the digits in positions $\ell+1,\ell+2,\dots$ after the decimal
point all being replaced by zeros. Then
$|x - \bar{x}| \leq b^{-\ell} \leq h^{-1} r^{-m-k-1}$. Thus, using
derivatives and the inequality $|y^2-z^2| \leq |y+z|\cdot |y-z|$, we have
\begin{eqnarray}
& & \left| \left|\sum_{j=m+1}^{m+k} e(r^j h x)\right|^2 - \left|\sum_{j=m+1}^{m+k} e(r^j h \bar{x})\right|^2 \right| \nonumber\\
& \leq & \left( \left|\sum_{j=m+1}^{m+k} e(r^j h x)\right| + \left|\sum_{j=m+1}^{m+k} e(r^j h \bar{x})\right| \right) \cdot \sum_{j=m+1}^{m+k} \left|e(r^j h x)- e(r^j h \bar{x})\right| \nonumber\\
& \leq & 2 k \sum_{j=m+1}^{m+k} 2 \pi r^j h |x - \bar{x}| \nonumber\\
& \leq & 2 k \sum_{j=m+1}^{m+k} 2 \pi r^j r^{-m-k-1} \nonumber\\
& \leq & c_1 k, \label{c1}
\end{eqnarray}
for some constant $c_1$ depending on $r$ and $h$. Let $x_v \in T_p$.  Then
by construction of the measure $\mu$ we have
$$
\mu \left( \left[x_v, x_v + b^{-\ell} \right) \right) = b^{-\ell/2},
$$
since those $x \in [0,1]$ for which
$\tau_p (x) \in [x_v, x_v + b^{-\ell} )$ form an interval of length
$b^{-\ell/2}$ (recall that we assumed that $\ell$ is even). Together with
\eqref{c1} this implies
\begin{eqnarray*}
& & \int_{[x_v, x_v + b^{-\ell} )} \left| \sum_{j=m+1}^{m+k} e(r^j h x) \right|^2~d\mu(x) \\
& \leq & \int_{[x_v, x_v + b^{-\ell} )} \left| \sum_{j=m+1}^{m+k} e(r^j h x_v) \right|^2 + c_1 k ~d\mu(x)\\
& = & \left(\left| \sum_{j=m+1}^{m+k} e(r^j h x_v) \right|^2 + c_1 k \right) \int_{[x_v, x_v + b^{-\ell} )} ~d\mu(x) \\
& = & b^{-\ell/2} \left( \left| \sum_{j=m+1}^{m+k} e(r^j h x_v) \right|^2 + c_1 k \right).
\end{eqnarray*}
Since
$$
\int_0^1 \left| \sum_{j=m+1}^{m+k} e(r^j h x) \right|^2~d\mu(x) = \sum_{v \in T_P(\ell)} \int_{[x_v, x_v + b^{-\ell} )} \left| \sum_{j=m+1}^{m+k} e(r^j h x) \right|^2~d\mu(x)
$$
we obtain
\begin{equation} \label{task}
\int_0^1 \left| \sum_{j=m+1}^{m+k} e(r^j h x) \right|^2~d\mu(x) \leq c_1 k + b^{-\ell/2} \sum_{v \in T_P(\ell)} \left| \sum_{j=m+1}^{m+k} e(r^j h x_v) \right|^2,
\end{equation}
and the main task for the proof of the lemma will be to estimate the sum on
the right of \eqref{task}.

Let 
\[
A(x,h,r,m,k)=  \left|\sum_{j=m+1}^{{m+k}} e(r^j h x)\right|^2.
\]
Since in the sequel the values $r$ and $h$ are fixed and we will always use
the expression with variables $m$ and $k$ we abbreviate $A(x,h,r,m,k)$ by
$A(x)$. We can rewrite the sum on the right-hand side of \eqref{task} as
\begin{eqnarray}
    \sum_{v\in T_P(\ell)}A(x_v)
& = &\sum_{v\in T_P(\ell)}
   ~ \sum_{j_1=m+1}^{{m+k}}
   ~ \sum_{j_2=m+1}^{{m+k}}
    e((r^{j_2}-r^{j_1})h x_v) \nonumber\\
& \leq & k b^{\ell/2} + 
    2 \sum_{i_1=1}^{k-1} \left| \sum_{v\in T_P(\ell)} \sum_{i_2=0}^{k-i_1-1} e \left(r^{m+1+i_2} (r^{i_1}-1)h x_v \right) \right|, \label{A_est}
\end{eqnarray}
where the term $k b^{\ell/2}$ comes from the contribution of the diagonal
$j_1 = j_2$, and where the summations in line \eqref{A_est} are obtained
from those in the line above by substituting $i_1 = |j_2-j_1|$ and then
first summing over all $j_1,j_2$ for which $|j_2-j_1|$ is fixed.

Since each sequence $v=v_1 v_2 \ldots v_\ell$ in $T_P(\ell)$ satisfies
$v_q=v_{2 q}$ for $q=1,\ldots, \ell/2$, for every integer $w$ we have
\begin{eqnarray}
\label{key}
\sum_{v\in T_P(\ell)}\;\; e(w x_v)&=&\sum_{v\in T_P(\ell)}  e\left(w \sum_{j=1}^\ell  v_j b^{-j}\right)
\\ \nonumber
&=&\sum_{v\in T_P(\ell)}  e\left(w  \sum_{\substack{q= 1 \\ q\text{ odd}}}^\ell v_q   
\left(\sum_{k=0}^{\lfloor \log_2 (\ell/q)\rfloor}  b^{-q 2^k}\right)\right)
\\\nonumber
&=&\sum_{v\in T_P(\ell)}
\prod_{\substack{q= 1 \\ q\text{ odd}}}^{\ell}
\prod_{k= 0}^{\lfloor \log_2 (\ell/q)\rfloor}  e\left(w v_q b^{-q2^k}\right)
\\\nonumber
&=&\prod_{\substack{q= 1 \\ q\text{ odd}}}^{\ell}
\sum_{u=0}^{b-1}
\prod_{k= 0}^{\lfloor \log_2 (\ell/q) \rfloor}  e\Big(u w b^{-q2^k}\Big)
\\
&=&\prod_{\substack{q= 1 \\ q\text{ odd}}}^{\ell}
\sum_{u=0}^{b-1}
  e\big(u w M_q\big), \label{wsums_in}
\end{eqnarray}
where 
\begin{equation*}
\label{eq:M}
M_q=  \sum_{k=0}^{\lfloor\log_2 (\ell / q)\rfloor}  b^{-q 2^k}. 
\end{equation*}

The internal $\sum_{u=0}^{ b-1} e(uy) $ for some real $y$ can be bounded by 
\begin{eqnarray} 
\left| \sum_{u=0}^{ b-1} e(uy) \right|
&=& \left| \sum_{\substack{0\leq u \leq b-2\\ \text{$u$ even}}} e(uy)(1+e(y)) \right| \nonumber\\
& \leq & \frac{b}{2} \Big|1+e(y)\Big| \qquad \textrm{if $b$ is even,} \label{sumb1}
\end{eqnarray}
and
\begin{eqnarray} 
\left| \sum_{u=0}^{ b-1} e(uy) \right| & \leq & 1 + \left| \sum_{\substack{1 \leq u \leq b-2\\ \text{$u$ odd}}} e(uy)(1+e(y)) \right| \nonumber\\
& \leq & 1 + \frac{b-1}{2} \Big|1+e(y)\Big| \qquad \textrm{if $b$ is odd.} \label{sumb2}
\end{eqnarray}
Since $|1 + e(y)|=2 |\cos \left(\pi y\right)|$, and since the term in
\eqref{sumb2} is larger than that in \eqref{sumb1}, for all $b$ we have the
upper bound
$$
\left| \sum_{u=0}^{b-1}
   e\big(u w M_q \big) \right| \leq 1 + (b-1) \big|\cos \left(\pi w  M_q\right)\big|
$$
for the sums appearing in line \eqref{wsums_in}.

Thus for (\ref{key}) we have the estimate
\begin{eqnarray}
\left| \sum_{v\in T_P(\ell)}\;\;e(w x_v) \right| & \leq & b^{\ell/2} 
\prod_{\substack{q= 1 \\ q\text{ odd}}}^{\ell} \left(\frac{1}{b} + \frac{b-1}{b} \left|\cos \left(\pi w  M_q \right) \right|\right) \nonumber\\
& \leq & b^{\ell/2} \prod_{\substack{\ell/2 \leq q \leq \ell, \\ q\text{ odd}}} \left(\frac{1}{b} + \frac{b-1}{b} \left|\cos \left( \pi w  M_q \right)\right|\right) \nonumber\\
& = & b^{\ell/2} \prod_{\substack{\ell/2 < q \leq \ell, \\ q\text{ odd}}} \left( \frac{1}{b} + \frac{b-1}{b} \left|\cos \left( \pi w  b^{-q} \right) \right|\right), \label{will_use}
\end{eqnarray}
where we used the crucial fact that for all $q$ satisfying $\ell/q<2$ we
have $\lfloor\log_2 (\ell / q)\rfloor=0$, and thus $M_q = b^{-q}$.  Note
that we were allowed to simply remove some of the factors when changing
from the first to the second line of the displayed formula, since all
factors are trivially bounded by~$1$.

We will use \eqref{will_use} with $w = r^{m+1+i_2} (r^{i_1}-1) h$, 
with the ranges of $i_1$ and $i_2$ specified in~\eqref{A_est}. 
By our choice of $\ell$ for such $w$ we have
\begin{eqnarray*}
w b^{-\ell} & \leq & r^{m+k} h b^{-\ell} \\
& \leq & r^{m+k} h b^{-((m+k+1) \log_b r) - \log_b h} \\
& \leq & r^{-1} \leq \frac{1}{2}.
\end{eqnarray*}
Thus for such $w$ we have
$$
\prod_{\substack{q > \ell, \\ q\text{ odd}}} \left(\frac{1}{b} + \frac{b-1}{b} \left|\cos \left( \pi w  b^{-q} \right) \right|\right) \geq \prod_{z=1}^\infty \left(\frac{1}{b} + \frac{b-1}{b} \left|\cos \left( \pi 2^{-z} \right) \right|\right) \geq c_2
$$
for some constant $c_2>0$, and thus for such $w$ the expression in line
\eqref{will_use} is bounded by
$$
c_2^{-1} \prod_{\substack{q > \ell/2, \\ q\text{ odd}}} \left(\frac{1}{b} + \frac{1}{b-1} \left|\cos \left(\pi w  b^{-q} \right) \right|\right).
$$

When we plug this estimate into \eqref{A_est} we obtain
\begin{eqnarray}
&& \left| \sum_{v\in T_P(\ell)}A(x_v) \right| \leq k b^{\ell/2} + \nonumber\\
& & \qquad + 2 c_2^{-1} b^{\ell/2} \sum_{i_1=1}^{k-1} \sum_{i_2=0}^{k-i_1-1} \prod_{\substack{q > \ell/2, \\ q\text{ odd}}} \left( \frac{1}{b} + \frac{b-1}{b} \left|\cos \left(\pi r^{m+1+i_2} (r^{i_1}-1) h  b^{-q} \right) \right|\right). \label{prod_esti}
\end{eqnarray}
We apply Lemma~\ref{lemma:cosine} to estimate the sums of products in this
formula, and using the lemma with $L = r^{m+1} (r^{i_1}-1) h$ we obtain
\begin{equation} \label{using_f}
\sum_{i_2=0}^{k-i_1-1} \prod_{\substack{q > \ell/2, \\ q\text{ odd}}} \left( \frac{1}{b} + \frac{1}{b-1} \left| \cos \left(\pi r^{m+1+i_2} (r^{i_1}-1) h  b^{-q} \right) \right|\right) \leq 2 k^{1-c_3}
\end{equation}
for a constant $c_3>0$. Note that for the application of the lemma it was
essential to assure that $L \geq b^{\ell/2}$, which with our choice of $L$
is $r^{m+1} (r^{i_1}-1) h \geq b^{\ell/2}$. However, this is true, since by
the assumption $m \geq k + 1 + 2 \log_r b$ and our choice of
$\ell \leq (((m+k+1) \log_b r) + \log_b h)+2$ we have
\begin{eqnarray}
r^{m+1} (r^{i_1}-1) h & \geq & h r^{m} \nonumber\\
& \geq & h r^{(m+k+1+2 \log_r b)/2} \nonumber\\
& = & h b^{((m+k+1) \log_b r)/2} b^{1} \nonumber\\
& \geq & b^{((m+k+1) \log_b r)/2 + (\log_b h) + 1} \nonumber\\
& \geq & b^{\ell/2}.
\end{eqnarray}
These formulas show where the difficulties come from in our setting, as
compared to Cassels' and Schmidt's work. Since we cannot control those
terms in the product where $M_q$ is complicated, we have to restrict the
product to relatively large values of $q$, where we have the simple
situation that $M_q = b^{-q}$.  However, since for this reason it is
necessary that our product starts at a large value of $q$, in order to be
able to apply Lemma~\ref{lemma:cosine} we have to make sure that the
frequencies (denoted by $L$ in the lemma) are large, which in turn requires
that the summation in Lemma~\ref{lemma:schmidt} cannot start at $1$, but
only at a value of $j$ which is relatively large in comparison with the
summation range $k$.

Using \eqref{using_f} for \eqref{prod_esti}, we obtain
\[
\left| \sum_{v\in T_P(\ell)}A(x_v) \right| \leq  b^{\ell/2} (k + 4 c_2^{-1} k^{2-c_3}).
\]
Combining this with \eqref{task} we finally obtain
\[
\int_0^1 \left| \sum_{j=m+1}^{m+k} e(r^j h x) \right|^2~d\mu(x) \leq (c_1 k + k + 4c_2^{-1} k^{2-c_3}) \leq k^{2-c_4}
\]
for a constant $c_4>0$ and all sufficiently large $k$.
\end{proof}

\begin{proof}[Proof of Theorem~\ref{thm:schmidt}]
Let  ${\mathcal R}$ be the set of real numbers in the unit interval whose expansion in base $b$ 
is in $T_P$, 
\[
 {\mathcal R}=\{ x_v :  v\in T_P\}.
\]
Note that by construction $\mu$ is supported on $\mathcal{R}$, so
$\mu([0,1] \backslash \mathcal{R})=0$. Consider an integer $r$
multiplicatively independent to $b$. 
To prove that $\mu$-almost all elements of
$\mathcal{R}$ are normal to base $r$, by Weyl's criterion we have to show
that for $\mu$-almost all $x \in [0,1]$ we have
$$
\lim_{N \to \infty} \frac{1}{N} \sum_{n=1}^N e(r^n h x) = 0
$$
for all integers $h>0$.\\ 

For $k \geq 1$, set $m_k = \left\lceil e^{\sqrt{k}} \right\rceil$. 
Furthermore, we define $M_0 = 0$ and
$$
M_k = m_1 + \dots + m_k, \qquad k \geq 1.
$$
For a fixed positive integer $h$, we define sets 
\[
E_k = \left\{ x \in [0,1]:~ \frac{1}{m_k} \left| \sum_{n=M_{k-1}+1}^{M_k} e(r^n h x) \right| > 1/k \right\}.
\]
The summation has $m_k$ terms  and, 
 for sufficiently large $k$, we have $$m_k \leq M_{k-1} + 1 + 2 \log_r b.$$
To  see this, notice that $m_k \approx e^{\sqrt{k}}$ and
$ \sqrt{k} e^{\sqrt{k}} < M_k$.
Thus, for all sufficiently large~$k$ we can apply Lemma~\ref{lemma:schmidt}, and by an
application of Chebyshev's inequality we have
\[
\mu(E_k) \leq \frac{k^2}{m_k^{c}},
\]
where $c>0$ is the constant from the conclusion of 
Lemma~\ref{lemma:schmidt}. By the rapid growth 
of the sequence $(m_k)_{k \geq 1}$ this implies
$$
\sum_{k=1}^\infty \mu(E_k) < \infty,
$$
and thus by the first Borel--Cantelli lemma $\mu$-almost surely only
finitely many events $E_k$ occur, so that in particular $\mu$-almost surely
we have
$$
\frac{1}{m_k} \left| \sum_{n=M_{k-1}+1}^{M_k} e(r^n h x) \right| \to 0
\qquad \text{as $k \to \infty$}.
$$
It is easily seen that this also implies that $\mu$-almost surely
$$
\lim_{k \to \infty} \frac{1}{M_k} \left|\sum_{n=1}^{M_k} e(r^n h x) \right| = 0.
$$
Finally, by the sub-exponential growth of $(M_k)_{k \geq 1}$, for all sufficiently large $N$ there is a value of $k$ such that $|N - M_k| = o(N)$. This implies 
\begin{equation} \label{weyl_c}
\lim_{N \to \infty} \frac{1}{N} \left|\sum_{n=1}^{N} e(r^n h x) \right| = 0, \qquad \text{$\mu$-almost surely}.
\end{equation}
Clearly, there are only countably many possible values of $h$ and $r$.
Thus $\mu$-almost all numbers $x \in [0,1]$ have the property
that~\eqref{weyl_c} is true for all positive integers $h$ and for all
integers $r \geq 2$ which are multiplicatively independent of~$b$.  This
proves the theorem.
\end{proof}

\section{Proof of Theorem~\ref{thm:theorem_D}} \label{sec:theorem_D}

We start with the first part of the theorem. It turns out that it is sufficient to relax the 
conclusion of the theorem to simple normality.

\begin{lemma} \label{lemma_D} Assume that for every $n \geq 1$ the random
  variable $X_n$ is uniformly distributed on $\{0, \dots, b-1\}$.  Assume
  furthermore that there exists a function
  $g:~\mathbb{N} \mapsto \mathbb{R}$ which is monotonically increasing to
  $\infty$ such that for all sufficiently large $n$ the random variables
\begin{equation}
\label{lemma-indep}
X_n, ~X_{n+1}, ~\dots, ~X_{n+ \lceil g(n) \log \log n \rceil}
\end{equation}
are mutually independent. Let $x$ be the real number whose  expansion in base $b$ is given 
by $x=0.X_1 X_2 X_3 \dots$. Then $\mathbb{P}$-almost surely the number $x$ is simply 
normal to base $b$.
\end{lemma}

We take Lemma~\ref{lemma_D} for granted, and show that it implies 
the first part of Theorem~\ref{thm:theorem_D}. 
Normality of a number to base~$b$ is equivalent to simple normality to all
the bases $b,b^2,b^3,\dots$. Two consecutive random variables such as
$X_1,X_2$ with values in $\{0, \dots, b-1\}$ define one random variable
$Y_1$ with value in $\{0, \dots, b^2-1\}$ by setting $Y_1 = b X_1 + X_2$.
Similarly $X_3$ and $X_4$ define $Y_2$, etc. Then the number
$(0.X_1 X_2 \dots)_b$ has base-$b^2$ expansion $(0.Y_1 Y_2 \dots)_{b^2}$.
Furthermore, since by the assumption of Theorem~\ref{thm:theorem_D} the
random variables
$$
X_{2n-1}, \dots, X_{2n-1+ \lceil g(2n-1) \log \log (2n-1) \rceil}
$$ 
are independent for sufficiently large $n$, this implies that there is a
function $\hat{g}$ which is monotonically increasing to $\infty$ such that
the random variables
$$
Y_n, \dots, Y_{n+ \lceil \hat{g}(n) \log \log n \rceil}
$$ 
are independent as well for all sufficiently large $n$. So the sequence
$Y_1, Y_2, \dots$ also satisfies the assumptions of Lemma~\ref{lemma_D}.
Accordingly, if Lemma~\ref{lemma_D} is true, then almost surely the number
$(0.Y_1 Y_2 \dots)_{b^2} = (0.X_1 X_2 \dots)_b$ is simply normal to base
$b^2$.  In the same way we can show that Lemma~\ref{lemma_D} implies that
almost surely $(0.X_1 X_2 \dots)_b$ is simply normal in bases $b^3$, that
almost surely it is simply normal to base $b^4$, and so on.  Since there
are only countably many bases $b,b^2,b^3,b^4,\dots$, it means that almost
surely $(0.X_1 X_2 \dots)_b$ is simply normal to all bases
$b,b^2,b^3,\dots$, which is equivalent to saying that almost surely
$(0.X_1 X_2 \dots)_b$ is normal to base $b$. Thus, to establish the first
part of Theorem~\ref{thm:theorem_D} it is sufficient to prove
Lemma~\ref{lemma_D}.

\begin{proof}[Proof of Lemma~\ref{lemma_D}]
Let $u \in \{0, \dots, b-1\}$ be a digit. Assume that $\varepsilon>0$ is fixed. Let $\theta = 1+\varepsilon$, and for $j \geq 1$ define
$$
N_j = \{n \geq 1:~\theta^{j-1} \leq n < \theta^j \}.
$$
We partition $N_j$ into disjoint sets of
$m_j = \lceil \varepsilon^{-2} \log j \rceil$ consecutive integers, where
for simplicity of writing we assume that $m_j$ divides $\# N_j$ (so that we
do not need to use one set of smaller cardinality at the end). We denote
these sets by $M_1^{(j)}, \dots, M_{r(j)}^{(j)}$, where
$r(j)= \#N_j / m_j$. It is easily verified that 
condition \eqref{lemma-indep}
 implies that all the random variables
$\{X_n:~n \in M_i^{(j)} \}$ are mutually independent, for all
$i \in \{1, \dots, r(j)\}$, provided that $j$ is sufficiently large; this
follows from the fact that the block length $m_j$ is of order roughly
$\varepsilon^{-2} \log \log n$ for $n \in N_j$, while by assumption
independence holds for random variables whose indices are within distance
$g(n) \log\log n$ of each other, where $g(n) \to \infty$.

By Hoeffding's inequality (see for example~\cite[Theorem 2.16]{bercu}) we have
\begin{eqnarray}
\mathbb{P} \left(\left| \frac{1}{m_j} \sum_{n \in M_i^{(j)}} \mathbf{1}(X_n = u)  - \frac{1}{b} \right| > \varepsilon \right) & \leq & 2 e^{-2 \varepsilon^2 m_j} \nonumber\\
& \leq & \frac{2}{j^{2}}, \label{calc_in}
\end{eqnarray}
for all $i \in \{1, \dots, r(j)\}$ and all sufficiently large $j$, where we
used $m_j \geq \varepsilon^{-2} \log j$. Here, and in the sequel, we write
$\mathbf{1}(E)$ for the indicator function of an event $E$. Let $Z_{i,j}$
denote the random variable
$\frac{1}{m_j} \sum_{n \in M_i^{(j)}} \mathbf{1}(X_n = u)  -
\frac{1}{b}$, for $i \in \{1, \dots, r(j)\}$. Then trivially
$|Z_{i,j}| \leq 1$. Thus we have
\begin{align*}
\mathbb{E} \big( |Z_{i,j}| \cdot \mathbf{1} \left(|Z_{i,j}| > \varepsilon\right) \big) & \leq\  \mathbb{E} \left(\mathbf{1} \left(|Z_{i,j}| > \varepsilon\right) \right) 
\\& =\  \mathbb{P} (|Z_{i,j}| > \varepsilon) 
\\&\leq\ \frac{2}{j^{2}}
\end{align*}
for sufficiently large $j$, as calculated in \eqref{calc_in}. By linearity of the expectation, this implies
$$
\mathbb{E} \left( \sum_{i=1}^{r(j)} \Big( |Z_{i,j}| \cdot \mathbf{1} \left(|Z_{i,j}| > \varepsilon\right)\Big) \right) \leq \frac{2 r(j)}{j^{2}},
$$
and thus by Markov's inequality
\begin{equation} \label{mark}
\mathbb{P} \left(\sum_{i=1}^{r(j)} |Z_{i,j}| \cdot \mathbf{1} \left(|Z_{i,j}| > \varepsilon \right) > \varepsilon r(j) \right) \leq \frac{2}{\varepsilon j^2}. 
\end{equation}
Note that
$$
|Z_{i,j}| \leq \varepsilon + |Z_{i,j}| \cdot \mathbf{1} \left(|Z_{i,j}| > \varepsilon\right).
$$
Thus
$$
\sum_{i=1}^{r(j)} |Z_{i,j}| \leq \varepsilon r(j) + \sum_{i=1}^{r(j)} |Z_{i,j}| \cdot \mathbf{1} \left(|Z_{i,j}| > \varepsilon\right),
$$
and \eqref{mark} implies that
\begin{equation} \label{prob}
\mathbb{P} \left( \sum_{i=1}^{r(j)} |Z_{i,j}|  >  2 \varepsilon r(j) \right) \leq \frac{2}{\varepsilon j^2}.
\end{equation}
Note that if $\sum_{i=1}^{r(j)} |Z_{i,j}|  \leq  2 \varepsilon r(j)$, then 

\begin{equation} \label{conv_nj}
 \left|\frac{1}{\#N_j} \# \{n \in N_j:~X_n = u \} - \frac{1}{b} \right| \leq 2 \varepsilon.
\end{equation}
The exceptional probabilities in \eqref{prob} form a convergent series when
summing over~$j$, so by the first Borel-Cantelli lemma with probability one
only finitely many of the corresponding events occur. Accordingly,
$\mathbb{P}$-almost surely we have \eqref{conv_nj} for all sufficiently
large $j$. From this it is easy to see that
\begin{equation*} 
 \left| \frac{1}{N}\#\{n \leq N:~X_n = u \} - \frac{1}{b} \right| \leq 4 \varepsilon
\end{equation*}
almost surely, for all sufficiently large $N$, where it is important that $N_{j+1} \approx (1+\varepsilon) N_j$. 
Since we can choose $\varepsilon$ arbitrarily close to zero, this proves Lemma~\ref{lemma_D}.
\end{proof}

\begin{proof}[Proof of the second part of Theorem~\ref{thm:theorem_D}]
  For the second part of the theorem, let $(Z_{j,m})_{j \geq 1,m \geq 0}$
  be an array of independent, identically distributed (i.i.d.) random
  variables having uniform distribution on $A=\{0, \ldots b-1\}$. Clearly
  it is sufficient to prove the second part of Theorem~\ref{thm:theorem_D}
  for those values of $K$ which are a positive (integral) powers of $2$, so
  we will assume that $K$ is of this form. 

  We define the digits $X_n$ of a
  number $x = (0.X_{1} X_{2} X_{3} \dots)_b$ by setting $X_n = Z_{j,m}$,
  where $j=\lfloor \log_2 n \rfloor$, and where $m$ is the unique integer
  for which $n \equiv m \mod 2Kr$ and $r$ is defined as the largest
  positive integer which is a power of 2 and for which $2^{(2^{r})} \leq n$
  (this definition only works for $n \geq 4$, so we may set
  $X_1 X_2 X_3 = 000$). So, for example when $K=1$, then for
  $n \in \{16,\dots,31\}$ we have $j=4$ and $r=2$, and thus
  $X_{16}, X_{17}, \dots, X_{31}$ is the pattern
  $Z_{4,0} Z_{4,1} Z_{4,2} Z_{4,3}$, being repeated four times. Or when
  $n \in \{2^{16},\dots, 2^{17} -1\}$, then $j=16$ and $r=4$, and so
  $X_{2^{16}}, \dots, X_{2^{17}-1}$ is the pattern
  $Z_{16,0}, Z_{16,1}, \dots, Z_{16,8}$, being repeated
  $2^j / (2rK) = 2^{16}/8$ times. This example is constructed in such a way
  that the digits $X_n, \dots, X_{n+2Kr-1}$ are mutually independent, where
  by definition $2^{(2^r)} \geq n$ and thus
  $$
  r \geq (\log_2 \log_2 n)/2 \geq 2(\log \log n)/3,
  $$
  which shows 
  that~\eqref{disjoint} indeed holds for all sufficiently large $n$.

  We will now show that the random number $x$ almost surely is not simply
  normal, which will follow from the fact that there is a digit $u \in A$
  for which the ratio
$$
\frac{1}{N} \Big|\{n \leq N:~X_n  = u\} \Big|
$$ 
does not converge to $1/b$. For simplicity, assume that $u=0$. For
$j \geq 1$, set $N_j = \{2^{j}, \dots, 2^{j+1}-1\}$. From the construction
of our sequence $(X_n)_{n \geq 1}$ it is easily seen that (for sufficiently
large $j$) the block of digits $X_{2^j}, \dots, X_{2^{j+1}-1}$ consists of
the block $Z_{j,0} \dots Z_{j,2Kr-1}$ for some appropriate value of
$r=r(j)$, which is repeated $2^j/(2Kr)$ times. Thus we have
\begin{equation} \label{zmi}
\# \Big\{n \in N_j:~X_n  = 0 \Big\} = \frac{2^j}{2Kr} ~\#  \Big\{0 \leq m < 2Kr:~Z_{j,m} = 0 \Big\}.
\end{equation}
We choose a ``small'' fixed value of $\varepsilon > 0$, and define events 
$$
E_j = \left\{ \left|\frac{1}{2^j} \cdot \# \left\{n \in N_j:~X_n  = 0 \right\} - \frac{1}{b} \right| \geq \varepsilon \right\}.
$$
By \eqref{zmi} we have
$$
\mathbb{P} (E_j) = \mathbb{P} \left\{ \left| \frac{1}{2Kr} \cdot \# \{0 \leq m < 2Kr:~Z_{j,m} = 0\} - \frac{1}{b} \right| \geq \varepsilon \right\}. 
$$
To estimate $\mathbb{P}(E_j)$, note that we have
\begin{equation} \label{distr} \nonumber
\# \Big\{0 \leq  m < 2Kr:~Z_{j,m} = 0 \Big\} = \sum_{m=0}^{2Kr-1} \mathbf{1}(Z_{j,m}=0),
\end{equation}
where $\mathbf{1}(Z_{j,m}=0)$ is the indicator function of the event
$Z_{j,m}=0$, and where accordingly
$(\mathbf{1}(Z_{j,m}=0))_{0 \leq m \leq r(j)-1}$ is a sequence of i.i.d.\
random variables with mean $1/b$ and variance
$\sigma_b^2 = \frac{1}{b}\left(1-\frac{1}{b}\right)$ (where we use the fact
that the independence property of the indicators $\mathbf{1}(Z_{j,m}=0)$ is
inherited from the independence assumption on the
$Z_{j,m}$'s). Accordingly, $\sum_{m=0}^{2Kr-1} \mathbf{1}(Z_{j,m}=0)$ has
binomial distribution $B(r,1/b)$. Using standard estimates for the tail
probabilities of the binomial distribution we can show that
$\mathbb{P}(E_j)$ is of order roughly $e^{-c(\varepsilon) 2Kr}$ for large $r$, where
$c(\varepsilon) = (1+o(1)) \varepsilon^2 \sigma_b^2/2$ as $\varepsilon \to 0$; this can be deduced either from lower bounds for the
tail of the binomial distribution, together with a linearization of the
Kullback-Leibler distance (see~\cite[Lemma 4.7.2]{ash} or~\cite[Theorem
11.1.3]{cover}), or from a comparison of the tail of a binomial
distribution with the tail of the normal distribution
(see~\cite{slud}). Note in particular that $c(\varepsilon)$ goes to zero as
a function of $\varepsilon$. Thus, if $\varepsilon$ was chosen so small
that $2K c(\varepsilon) < 1/8$, then we certainly have
$$
\mathbb{P} (E_j) \geq e^{-r/4}
$$
for all sufficiently large $j$. By the definition of $r$ we have 
$r=r(j) \geq \frac{\log_2 (j+1)}{2} \geq \frac{\log j}{2}$ for all sufficiently large $j$. 
Thus,
$$
\mathbb{P} (E_j) \geq e^{-r/4} \geq e^{-(\log j)/2} = \frac{1}{\sqrt{j}}
$$
for all sufficiently large $j$. This allows us to deduce that 
$$
\sum_{j=1}^\infty \mathbb{P}(E_j) = + \infty.
$$
Note that the events $(E_j)_{j \geq 1}$ are mutually independent, since $E_j$ only 
depends on random variables $Z_{j,m}$ whose first index is $j$. 
Thus by the second Borel--Cantelli lemma with probability one infinitely many events $E_j$ occur. 
However, this means that with probability one the digit $0$ does not have the correct asymptotic 
frequency within the blocks of digits with indices in $N_j$. Note that the length of the blocks $N_j$ grows so quickly that
$$
\lim_{j \to \infty} \frac{\#N_j}{\#N_1 + \cdots + \#N_{j-1}} = 1 > 0,
$$
so that the contribution of digits contained in $N_j$ is not (asymptotically) negligible in comparison 
with the contribution of the digits from all the previous blocks. From this it is easy to deduce that
$$
\frac{1}{N} \cdot \# \Big\{n \leq N:~X_n  = 0 \Big\} \not\to \frac{1}{b}
$$
for $\mathbb{P}$-almost all $x$, which proves the second part of the theorem.
\end{proof}

\section{Proof of Theorem~\ref{thm:toeplitz}} \label{sec:toeplitz}

In this section, finite sequences of digits are called \emph{words}.
  If $w = a_1\cdots a_n$ is a word of length~$n$ and $i,j$ are two integers
  such that $1 \le i \le j \le n$, the word $a_i \cdots a_j$ is called
  either the \emph{block} of~$w$ from position~$i$ to position~$j$ or the
  \emph{block} of length $j-i+1$ at position~$i$ in~$w$.
Let's fix an integer~$k$ and consider the finite set $I$ of integers
$I = \{i: 1\leq i\leq (p_1p_2)^{k+1}\}$ and its subset
$J = \{ j \in I : p_1^{k+1} | j \text{ or } p_2^{k+1} | j \}$.  The set~$J$
can be decomposed as $J = J_1 \cup J_2$ where
$J_1 = \{ p_1^{k+1}m : 1 \le m \le p_2^{k+1}\}$,
$J_2 = \{ p_2^{k+1}m : 1 \le m \le p_1^{k+1}\}$ and
$J_1 \cap J_2 = \{ (p_1p_2)^{k+1}\}$.  It follows that $J$ has cardinality
$p_1^{k+1} + p_2^{k+1} - 1$.  Let $\rho_J$ be the function which maps each
word~$w$ of length~$(p_1p_2)^{k+1}$ to the word of length
$(p_1p_2)^{k+1} -p_1^{k+1} - p_2^{k+1} + 1$ obtained by removing from~$w$
symbols at positions in~$J$. Formally, if the word~$w$ is equal to
$a_1\cdots a_{(p_1p_2)^{k+1}}$ which is written $\prod_{i \in I}{a_i}$, the
word~$\rho_J(w)$ is equal to $\prod_{i \in I \setminus J}{a_i}$. Note
 that
\begin{displaymath}
  (p_1p_2)^{k+1} -p_1^{k+1} - p_2^{k+1} + 1 = (p_1^{k+1}-1)(p_2^{k+1}-1)
\end{displaymath}

We introduce a last notation.  Let $\sigma$ be a permutation of
$\{1,\ldots,n\}$.  It induces a permutation, also denoted by~$\sigma$,
of~$A^n$ defined by
$\sigma(a_1\cdots a_n) = a_{\sigma(1)} \cdots a_{\sigma(n)}$ for each word
$a_1\cdots a_n$ of length~$n$.

\begin{lemma} \label{lem:toeplitz-key}
  Suppose that $\tau_P(x)$ is decomposed $\tau_P(x) = w_1w_2w_3\cdots$
  where each word~$w_i$ has length~$(p_1p_2)^{k+1}$.  There is a
  permutation~$\sigma$ of $\{1,\ldots,(p_1^{k+1}-1)(p_2^{k+1}-1)\}$ such
  that for each integer~$i$, the word~$\sigma(\rho_J(w_i))$ is equal to the
  concatenation 
$\prod_{ i_1=0}^{ k} \prod_{ i_2=0} ^{ k}u_{i,i_1,i_2}$
 of the
    $(k+1)^2$ blocks of~$x$ where each block $u_{i,i_1,i_2}$ starts at
    position $(p_1-1)(p_2-1)p_1^{i_1}p_2^{i_2}(i-1)+1$ and has length
    $(p_1-1)(p_2-1)p_1^{i_1}p_2^{i_2}$.
\end{lemma}
Note for each $0 \le i_1,i_2 \le k$, the length of $u_{i,i_1,i_2}$ is the
same for all $i \ge 1$ and that
\begin{displaymath}
   x  = u_{1,i_1,i_2}u_{2,i_1,i_2}u_{3,i_1,i_2} \cdots
\end{displaymath}
is the decomposition of~$x$ in blocks of
length~$(p_1-1)(p_2-1)p_1^{i_1}p_2^{i_2}$.  What is important is that both
$\rho_J$ and~$\sigma$ are fixed and do not depend on~$w_i$.  The
permutation~$\sigma$ is not made explicit by the statement of the lemma
because it is not used in the sequel but it can be easily recovered from
the proof of the lemma.
\begin{proof}
  The positions in~$\tau_P(x)$ of the word~$w_i$ are the integers from
  $(p_1p_2)^{k+1}(i-1)+1$ to $(p_1p_2)^{k+1}i$, 
  so they are of the form $(p_1p_2)^{k+1}(i-1)+j$ where $1 \le j \le (p_1p_2)^{k+1}$.  
  By definition, the function~$\rho_J$ removes the symbols at a position
  of the form $(p_1p_2)^{k+1}(i-1)+j$ where $j \in J$.  It follows that
  the word $\rho_J(w_i)$ contains the symbols at a position
  $(p_1p_2)^{k+1}(i-1)+j$ where $j$ can be written
  $j = p_1^{i_1}p_2^{i_2}m$ for $0 \le i_1,i_2 \le k$ and $m$ 
  divisible by neither $p_1$ nor~$p_2$.

  Fix $i_1$ and~$i_2$ such that $0 \le i_1,i_2 \le k$ and consider
  all symbols of~$w_i$ at positions of the form $(p_1p_2)^{k+1}(i-1)+j$
  where $j = p_1^{i_1}p_2^{i_2}m$ for $0 \le i_1,i_2 \le k$ and $m$ 
  divisible by neither $p_1$ or~$p_2$.  Note that in~$p_1p_2$ consecutive
  integers, exactly $p_1+p_2-1$ of them are divisible by either $p_1$
  or~$p_2$.  Since $1 \le m \le p_1^{k+1-i_1}p_2^{k+1-i_2}$, there are
  exactly $(p_1-1)(p_2-1)p_1^{k-i_1}p_2^{k-i_2}$ possible values for~$m$.
  By definition of the Toeplitz transform~$\tau_P$, these symbols
  are at consecutive positions in~$x$.   More precisely, they are the
  symbols from position $(p_1-1)(p_2-1)p_1^{k-i_1}p_2^{k-i_2}(i-1)+1$ 
  to position $(p_1-1)(p_2-1)p_1^{k-i_1}p_2^{k-i_2}i$ in~$x$.  Calling this
  word $u_{i,k-i_1,k-i_2}$ for each $0 \le i_1,i_2 \le k$ provides the 
  decomposition of~$\rho_J(w_i)$.
\end{proof}

We continue with two lemmas that show  that Condition~(\ref{toeplitz-ii}) is
quite robust.  First we prove that if Condition~(\ref{toeplitz-ii}) of
Theorem~\ref{thm:toeplitz} holds for  words of some given lengths then it also
holds for all shorter words.

\begin{lemma} \label{lem:shorter}

  Fix an integer $k \geq 0$ and let
  $(\ell_{i_1,i_2})_{0 \le i_1,i_2 \le k}$ be a family of non-negative
  integers.  If
  \begin{displaymath}
    \lim_{N\to\infty} \frac{1}{N}\#\left\{
      n :
      \begin{array}{l}
        1\leq n \le N, \text{ and for all } 0 \le i_1,i_2 \le k,\\
        u_{i_1,i_2} \text{ occurs in $x$ at position
                       $(p_1-1)(p_2-1)p_1^{i_1}p_2^{i_2}n$}
      \end{array} 
      \right\}
    \end{displaymath}
    is equal to $b^{ -\sum_{0 \le i_1, i_2 \le k} |u_{i_1,i_2 }|}$ for all
    families $(u_{i_1,i_2})_{0 \le i_1,i_2 \le k}$ of finite words such
    that $|u_{i_1,i_2}| = \ell_{i_1, i_2}$, then it also holds for all
    families $(v_{i_1, i_2})_{0 \le i_1, i_2 \le k}$ of finite words such
    that $|v_{i_1,i_2}| \leq \ell_{i_1,i_2}$.
\end{lemma}
\begin{proof}
  The result follows from the equality
  \begin{align*}
   & \#\left\{ n :
      \begin{array}{l}
        1\leq n \le N, \text{ and for all } 0 \le i_1,i_2 \le k,\\
        v_{i_1,i_2} \text{ occurs in $x$ at position
                               $(p_1-1)(p_2-1)p_1^{i_1}p_2^{i_2}n$}
      \end{array} 
    \right\} \\
    = \sum_{w_{i_1,i_2}}
   & \#\left\{ n :
      \begin{array}{l}
        1\leq n \le N, \text{ and for all } 0 \le i_1,i_2 \le k,\\
        v_{i_1,i_2}w_{i_1,i_2} \text{ occurs in $x$ at position
                              $(p_1-1)(p_2-1)p_1^{i_1}p_2^{i_2}n$}
      \end{array} 
    \right\} 
  \end{align*}
  where the summation ranges over all families
  $(w_{i_1, i_2})_{0 \le i_1, i_2 \le k}$ of finite words such that
  $|w_{i_1, i_2}| = \ell_{i_1,i_2}-|v_{i_1, i_2}|$ for each
  $0 \le i_1, i_2 \le k$.
\end{proof}

The following lemma shows that offsets can be freely added in
Condition~(\ref{toeplitz-ii}) of Theorem~\ref{thm:toeplitz}.
\begin{lemma} \label{lem:offsets}
  Fix a non-negative integer $k$ and let
  $(\delta_{i_1, i_2})_{0 \le i_1, i_2 \le k}$ be a family of
  integers. Then, for every family
  $(u_{i_1, i_2})_{0 \le i_1, i_2 \le k}$ of finite words
    \begin{displaymath}
      \lim_{N\to\infty} \frac{1}{N}\#\left\{ n :
      \begin{array}{l}
        1\leq n \le N, \text{ and for all } 0 \le i_1, i_2 \le k,\\
        u_{i_1,i_2} \text{ occurs in $x$ at position
                      $(p_1-1)(p_2-1)p_1^{i_1}p_2^{i_2}n$}
      \end{array} 
      \right\}
    \end{displaymath}
    is equal to $b^{ -\sum_{0 \le i_1,i_2 \le k} |u_{i_1, i_2 }|}$ if and
    only if for every family $(v_{i_1, i_2})_{0 \le i_1, i_2 \le k}$ of
    finite words
   \begin{displaymath}
      \lim_{N\to\infty} \frac{1}{N}\#\left\{ n :
      \begin{array}{l}
        1\leq n \le N, \text{ and for all } 0 \le i_1, i_2 \le k,\\
        v_{i_1,i_2} \text{ occurs in $x$ at position
        $(p_1-1)(p_2-1)p_1^{i_1}p_2^{i_2}n + \delta_{i_1,i_2}$}
      \end{array} 
      \right\}
    \end{displaymath}
    is equal to $b^{ -\sum_{0 \le i_1, i_2 \le k} |v_{i_1, i_2 }|}$.
\end{lemma}
Note that if some $\delta_{i_1,i_2}$ is negative, then the position
$(p_1-1)(p_2-1)p_1^{i_1}p_2^{i_2}n + \delta_{i_1,i_2}$ may not exist for
small values of~$n$ because it is negative.  However, this does not really
matter because we are considering a limit when $N$ goes to infinity.

\begin{proof}
  In order to replace $\delta_{i_1,i_2}$ by $\delta_{i_1,i_2}+1$ for a
  single pair $(i_1,i_2)$, it suffices to sum up for all possible
  symbols~$a$ in~$A$, all equalities of Condition~(\ref{toeplitz-ii}) for
  $av_{i_1,i_2}$ to get the equality of Condition~(\ref{toeplitz-ii}) for
  $v_{i_1,i_2}$ with $\delta_{i_1,i_2}+1$.

  In order to replace $\delta_{i_1,i_2}$ by
  $\delta_{i_1,i_2}-(p_1-1)(p_2-1)p_1^{i_1}p_2^{i_2}$ for all pairs
  $(i_1,i_2)$, it suffices to replace $n$ by $n-1$.  These two replacements
  allow us to get any possible change of offsets.
\end{proof}

We now come to the proof of Theorem~\ref{thm:toeplitz}.
\begin{proof}[Proof of Theorem~\ref{thm:toeplitz}]
  We first prove that Condition~\ref{toeplitz-i} implies
  Condition~\ref{toeplitz-ii}.  We suppose that $\tau_P(x)$ is normal.  By
  Lemma~\ref{lem:shorter}, it suffices to prove Condition~\ref{toeplitz-ii}
  when the length of each word $u_{i_1,i_2}$ is
  $(p_1-1)(p_2-1)p_1^{i_1+\ell}p_2^{i_2+\ell}$ for some fixed but arbitrary
  integer~$\ell$.  Consider the decomposition $\tau_P(x) = w_1w_2w_3\cdots$
  where each word~$w_i$ has length~$(p_1p_2)^{k+1}$ and let $y$ be the
  sequence $y = w'_1w'_2w'_3\cdots$ where each word~$w'_i$ is given by
  $w'_i = \rho_J(w_i)$.  Since $w'_i$ is obtained from~$w_i$ by removing
  symbols at fixed positions, the sequence~$y$ is also normal.  Here we use
  the result that selecting digits along a periodic sequence of positions
  preserves normality~\cite{Wall1949}.  By Lemma~\ref{lem:toeplitz-key},
  the symbols of each word~$w'_i$ can be rearranged by a
  permutation~$\sigma$ to the concatenation
  \begin{displaymath}
     \prod_{ i_1=0} ^{k} \prod_{ i_2=0}^{ k}{z_{i,i_1,i_2}}
  \end{displaymath}
  where $z_{i,i_1,i_2}$ is the block in~$x$ of length
  $(p_1-1)(p_2-1)p_1^{i_1}p_2^{i_2}$ that starts at position
\linebreak
  $(p_1-1)(p_2-1)p_1^{i_1}p_2^{i_2}(i-1)+1$.  It follows that the
  concatenation $w_iw_{i+1} \cdots w_{i+(p_1p_2)^{\ell}-1}$ of
  $(p_1p_2)^{\ell}$ such words can be also rearranged by another
  permutation to the concatenation
  \begin{displaymath}
    \prod_{ i_1=0}^{k}\prod_{ i_2=0}^{k}
     \tilde z_{i,i_1,i_2}    
  \end{displaymath}
  where $\tilde{z}_{i,i_1,i_2}$ is the block in $x$ of length
  $(p_1-1)(p_2-1)p_1^{i_1+\ell}p_2^{i_2+\ell}$ that starts at position
  $(p_1-1)(p_2-1)p_1^{i_1+\ell}p_2^{i_2+\ell}(i-1)+1$.

  Since $y$ is normal, all words of length
  $(p_1^{k+1}-1)(p_2^{k+1}-1)(p_1p_2)^{\ell}$ occur with the same frequency
  in~$y$ as a concatenation $w_iw_{i+1} \cdots w_{i+(p_1p_2)^{\ell}-1}$.
  It follows that for every family $(u_{i_1, i_2})_{0 \le i_1, i_2 \le k}$
  of finite words such that $u_{i_1,i_2}$ has length
  $(p_1-1)(p_2-1)p_1^{i_1+\ell}p_2^{i_2+\ell}$ for $0 \le i_1, i_2 \le k$,
  the limit
  \begin{displaymath}
    \lim_{N\to\infty} \frac{1}{N}\#\left\{ n :
    \begin{array}{l}
      1\leq n \le N, \text{ and for all } 0 \le i_1, i_2 \le k,\\
      u_{i_1,i_2} \text{ occurs in $x$ at position
      $(p_1-1)(p_2-1)p_1^{i_1}p_2^{i_2}n+1$}
    \end{array} 
    \right\}
  \end{displaymath}
  has the same value.  By Lemma~\ref{lem:offsets}, the offset $+1$ in the
  position $(p_1-1)(p_2-1)p_1^{i_1}p_2^{i_2}n+1$ can be removed and the
  result is established.

  We now prove the converse.  By a result of Long~\cite{Long1957}, a
    number~$x$ is normal to base~$b$ if and only if there exists positive
    integers $m_1 < m_2 < \cdots$ such that $x$ is simply normal to the
    bases~$b^{m_i}$, $i \ge 1$.  Therefore, it is sufficient to prove that,
    for infinitely many length~$k$, all words of length~$k$ have the
    expected frequency in~$\tau_P(x)$.  We claim that for each word $w$ of
    length $(p_1p_2)^{\ell}$ for any integer~$\ell$, the frequency of~$w$
    at positions multiple of $(p_1p_2)^{\ell}$ is the expected one, namely
    $|A|^{-(p_1p_2)^{\ell}}$. Now suppose that $\ell$ is fixed and let
  $k$ be an integer to be fixed later.  Factorize
  $\tau_P(x) = w_1w_2w_3\cdots$ where each word~$w_i$ has
  length~$(p_1p_2)^{k+1}$ and let $y$ be the sequence
  $y = w'_1w'_2w'_3\cdots$ where $w'_i = \rho_J(w_i)$.  By
  Lemma~\ref{lem:toeplitz-key} and by the hypothesis of
  Condition~\ref{toeplitz-ii}, the sequence~$y$ is normal.  For each
  integer~$i$, the word~$w_i$ is obtained from~$w'_i$ by inserting
  $p_1^{k+1} + p_2^{k+1} - 1$ symbols. Each word~$w_i$ contains
  $(p_1p_2)^{k+1-\ell}$ blocks of size $(p_1p_2)^{\ell}$.  These
  $p_1^{k+1} + p_2^{k+1} - 1$ inserted symbols can spoil at most
  $p_1^{k+1} + p_2^{k+1} - 1$ blocks of size $(p_1p_2)^{\ell}$ but this
  number of possible spoiled blocks becomes negligible with respect to the
  total number of such blocks when $k$ goes to infinity.  Hence
    normality follows by taking $k$ great enough. This concludes the proof
  that $\tau_P(x)$ is normal.
\end{proof}
\bigskip

\section*{Acknowledgements}

Aistleitner is supported by the Austrian Science Fund (FWF), projects Y-901
and F 5512-N26.  Becher and Carton are members of the Laboratoire
International Associ\'e INFINIS, CONICET/Universidad de Buenos
Aires--CNRS/Universit\'e Paris Diderot and they are supported by the ECOS
project PA17C04.  Carton is also partially funded by the DeLTA project
(ANR-16-CE40-0007).\\

This paper was initiated in November 2016 during the workshop ``Normal Numbers: Arithmetic, Computational and Probabilistic Aspects'' at the Erwin Schr\"odinger International Institute for Mathematics and Physics (ESI) in Vienna. We thank the ESI for bringing the three of us together, and for providing a stimulating atmosphere for mathematical discussions.\\

The authors are very grateful to the anonymous referee for reading the first version of this paper with exceptional accurateness and for making many suggestions for possible improvements. The comments of the referee helped significantly to improve the presentation of the results and proofs in this paper.


\bigskip
\bigskip

{\small
\begin{minipage}{\textwidth}
Christoph Aistleitner 
\\Institute of Analysis and Number Theory
\\Graz University of Technology, Austria
\\aistleitner@math.tugraz.at
\bigskip\\
\noindent
Ver\'onica Becher
\\
Departamento de  Computaci\'on, Facultad de Ciencias Exactas y Naturales \& ICC\\
 Universidad de Buenos Aires \&  CONICET, Argentina
\\
vbecher@dc.uba.ar
\bigskip\\
Olivier Carton
\\
Institut de Recherche en Informatique Fondamentale
\\Universit\'e Paris Diderot, France
\\olivier.carton@irif.fr
\end{minipage}
}

\end{document}